\documentclass[oneside,english]{amsart}
\usepackage{ae,aecompl}
\usepackage[T1]{fontenc}
\usepackage[latin9]{inputenc}
\usepackage{geometry}
\usepackage{amstext}
\usepackage{amsthm}
\usepackage{amssymb}
\usepackage{stmaryrd}
\usepackage{setspace}
\usepackage[authoryear]{natbib}
\makeatletter
\numberwithin{equation}{section}
\numberwithin{figure}{section}
\theoremstyle{plain}
\newtheorem{thm}{\protect\theoremname}
\theoremstyle{definition}
\newtheorem{defn}[thm]{\protect\definitionname}
\theoremstyle{plain}
\newtheorem{lem}[thm]{\protect\lemmaname}
\theoremstyle{plain}
\newtheorem{prop}[thm]{\protect\propositionname}
\theoremstyle{plain}
\newtheorem{cor}[thm]{\protect\corollaryname}
\theoremstyle{plain}
\newtheorem{fact}[thm]{\protect\factname}

\makeatother

\usepackage{babel}
\providecommand{\corollaryname}{Corollary}
\providecommand{\definitionname}{Definition}
\providecommand{\factname}{Fact}
\providecommand{\lemmaname}{Lemma}
\providecommand{\propositionname}{Proposition}
\providecommand{\theoremname}{Theorem}

\begin{document}
\title{Foundations with Imagination}
\author{Toby Meadows}
\begin{abstract}
We show that countable set theory, $ZFC^{-}+\forall x\ |x|\leq\omega$,
is unable to eliminate imaginaries. In other words, this theory cannot
provide representatives for arbitrary definable equivalence relations.
We also see that $ZFC^{-}$ and $ZFC^{-}+\exists\kappa(Inacc(\kappa)\wedge\forall x\ |x|\leq\kappa)$
also fail to eliminate imaginaries.
\end{abstract}

\maketitle
When working in a mathematical theory with foundational ambitions,
we frequently define equivalence relations. And it is often convenient
to be able to define arguably canonical representatives for each of
the underlying equivalence classes. Sometimes those representative
are members of those classes and sometimes they are not. When working
in Peano arithmetic, or any other theory with a definable well-ordering
of its domain, we may simply take the least element of each cell of
the equivalence relation. In $ZFC$, we do not always have such an
ordering,\footnote{Although note that if we add $V=HOD$ or $V=L$ to $ZFC$ we do get
a definable well-ordering of the universe. Moreover, $V=HOD$ is equivalent,
over $ZF$, to the existence of such a well-ordering. } however, we can make do with Scott's trick and use the subclass of
each cell consisting of those members that have minimal rank. In this
case, we do not get an element of the equivalence class, but we still
get an object, a set, which allows us to quantify over and thus, quite
freely theorize about the cells of the equivalence relation. Theories
that can do this are said to be able to \emph{eliminate imaginaries}.
Our goal in this paper is to show that the natural theory of countable
sets cannot eliminate imaginaries. This theory is formulated by removing
the powerset axiom from $ZFC$ and adding a new axiom stating that
every set is countable. We shall call this theory, \emph{countable
set theory}, and denote it as $ZFC_{count}^{-}$. It is a useful theory
in which a lot of mathematics can be done. For example it is well-known
to be equivalent to the theory of second order arithmetic and thus,
provides an alternative framework for the stronger axioms used in
reverse mathematics \citep{EnayatLelykFOCat,Simpson}. It follows
from our results that the commonly used, $ZFC^{-}$, also fails to
eliminate imaginaries.

We shall start in Section \ref{sec:Motivating-remarks} with some
basic definitions. Then in Section \ref{sec:A-warmup-theorem}, we'll
give a warmup theorem showing that second order arithmetic cannot
eliminate imaginaries. I'm grateful to Asaf Karagila for sharing the
key idea behind this claim. Finally, in Section \ref{sec:The-main-theorem},
we generalize the idea of Section \ref{sec:A-warmup-theorem} in order
to prove our main claim. 

\section{Motivating remarks\label{sec:Motivating-remarks}}

Let's start with the main definition. 
\begin{defn}
Let us say that a theory $T$ \emph{eliminates imaginaries}\footnote{See Section 4.4 in \citep{Hodges} for further information and discussion
of the history of this concept. In that setting, this is called the
\emph{uniform elimination of imaginaries}. However, in this foundational
setting it seems more convenient to elide the mention of uniformity
as the non-uniform version appears less interesting here.} if $T$ proves that some formula $E$ defines an equivalence relation,
then there is a formula $F$ such that $T$ proves that $F$ defines
a function that respects $E$; i.e., $T$ proves that for all $\bar{x},\bar{y}$,
$\bar{x}E\bar{y}$ iff $F(\bar{x})=F(\bar{y})$.
\end{defn}

In the definition above $\bar{x}$ is intended to represent a finite
sequence of variables $x_{0},...,x_{n}$. However, since we will be
working with theories that can represent finite sequences as objects,
we may ignore this aspect of things and work with the version of this
definition where we just have single variables $x$ and $y$. 

It is not difficult to see that $PA$, $ZF$ and their extensions
are all able to eliminate imaginaries. Moreover, this is a feature
that is frequently used in foundational contexts. For a couple of
classic examples, consider the use of Scott's trick: in defining cardinals
without choice in $ZF$; and in the definition of ultrapowers from
large cardinals in $ZFC$. The elimination of imaginaries also allows
us to avoid the use of quotients in the interpretation of other theories.\footnote{In a quotient interpretation, the identity relation is not preserved
and is rather interpreted as an equivalence class from the perspective
of the interpreting theory. } Or in the language of \citet{bHMoritaEquiv}, a failure to eliminate
imaginaries leads to a failure of Morita completeness and, in particular,
a failure to deliver quotient sorts. The ability to eliminate imaginaries
has also figured in a lot of recent work investigating properties
of theories related to categoricity. For example, building on work
of \citet{freire2020biinterpretation}, \citet{EnayatLelykFOCat}
have shown that internally categorical, sequential theories that eliminate
imaginaries are tight. Informally speaking, an \emph{internally categorical
}theory is a theory that can see that it only has one model up to
isomorphism.\footnote{It does this in a first order, rather than second order way, by duplicating
the language and adding the same axioms to the new language. Axiom
schemata then use formulae from the the combined language. For more
detail, see \citep{VaanZerm,MaddyVaananenCat}.} A theory is \emph{tight} if any pair of bi-interpretable extensions
of that theory must be the same theory, in the sense that they have
the same consequences. Less formally, a tight theory cannot be extended
in meaningfully different ways. Our proof below then establishes that
there is an internally categorical theory, $ZFC_{count}^{-}$, that
cannot eliminate imaginaries, but is still tight.\footnote{See Theorem 19 of \citep{meadows2025intcatgenericmult} for a proof
that $ZFC_{count}^{-}$ is internally categorical; and see Corollary
2.5.1 of \citep{enayat2017variations} or Theorem 29 of \citep{meadows2025intcatgenericmult}
for a proof that $ZFC_{count}^{-}$. Remark 8(b) in \citep{EnayatLelykFOCat}
then observes that solid theories are always tight.} Thus, the elimination of imaginaries is not required for internally
categorical theories to be tight.

\section{A warmup theorem\label{sec:A-warmup-theorem}}

Let $SOA$ be the theory of second order arithmetic with full comprehension
and choice for all definable sets of reals indexed by naturals.\footnote{See \citep{EnayatLelykFOCat} and \citep{Simpson} for more details.
In the latter, this theory is denoted as $\Sigma_{\infty}^{1}-AC_{0}$
noting that the definable choice principles end up implying the comprehension
principles. See Section VII.6 and Lemma VII.6.6 in \citep{Simpson}
for more information. Also note that $SOA$ is also often notated
as $Z_{2}$ while $PA$ is often denoted as $Z_{1}$ \citep{HilbertGrundlagen}.} As a warmup to our main result about $ZFC_{count}^{-}$, we aim to
show that $SOA$ cannot eliminate imaginaries using an automorphism
argument based on forcing. More specifically, we will define an equivalence
relation $E$ on $\mathbb{R}$ and then suppose toward a contradiction
that we can define $F:\mathbb{R}\to\mathbb{R}$ that respects $E$.
We shall then deliver some $g\in\mathbb{R}$ and an automorphism of
the underlying forcing that moves $g$ to some $g^{*}$ where $gEg^{*}$
but $F(g)\neq F(g^{*})$ giving us our desired contradiction. 

First we describe the required forcing and some helpful technical
vocabulary. We let $\mathbb{Q}=Add(\omega,\omega_{2})$, so $\mathbb{Q}$
consists of finite partial functions $p:\omega\times\omega_{2}\rightharpoondown2$
ordered by reverse inclusion. Given $p\in\mathbb{Q}$, let us say
that the $\omega$\emph{-domain} of $p$, abbreviated $\omega\text{-}dom(p)$,
is the set of $n$ such that $p(n,\alpha)$ is defined for some $\alpha<\omega_{2}$.
Or in more formal notation, we write 
\[
\omega\text{-}dom(p)=\{n\in\omega\ |\ p(n,\alpha)\downarrow\text{ for some }\alpha\in\omega_{1}\}.
\]
Let us then say that an automorphism $\sigma:\mathbb{Q}\cong\mathbb{Q}$
is \emph{$\omega$-based }if it is determined by a permutation of
the first coordinate, or $\omega$-domain, of conditions in $\mathbb{Q}$.
More specifically, $\sigma$ is such that there is some permutation
$\sigma:\omega\to\omega$ such that 
\[
\sigma(p)(\sigma(n),\alpha)=p(n,\alpha)
\]
whenever the right hand side is defined. The following lemma is then
the key to the claim. We shall use it to ensure that $F(g)\neq F(g^{*})$
below. 
\begin{lem}
(Karagila)\footnote{I'm grateful to Asaf for letting me include this proof here. The underlying
idea is obviously elegant. Any ugliness in its presentation is due
to me.} Suppose $G$ is $\mathbb{Q}$-generic over $V$. Let $x\subseteq V$
be such that $x\in V[G]\backslash V$. Then for all $p\in G$ and
$\dot{x}\in V^{\mathbb{Q}}$ with $\dot{x}_{G}=x$ there is some $\omega$
-based automorphism $\sigma:\mathbb{Q}\cong\mathbb{Q}$ with $\sigma\in V$
such that $\sigma(p)=p$ and $(\sigma\dot{x})_{G}\neq x$.\label{Karagila}
\end{lem}

\begin{proof}
Let $x$ and $\dot{x}$ be as assumed in the statement of the theorem.
Then since $x\subseteq V$ is not an element of $V$, $x$ must be
infinite and so there must be some $u\in x$ such that $p$ cannot
force that $u\in\dot{x}$. Thus, we may fix $q,r\leq p$ such that
$q\in G$
\[
q\Vdash u\in\dot{x}\ \text{and}\ r\Vdash u\notin\dot{x}.
\]
We assume without loss of generality that $\omega\text{-}dom(p)=m$
and $\omega\text{-}dom(q)=\omega\text{-}dom(r)=m+n$. Using the genericity
of $G$, we may fix a condition $r^{*}\in G$ whose $\omega$-domain
is $[k,k+n)$ for some $k>m+n$ and which is such that $r^{*}$ is
the translation of $r$ that shifts the part of $r$ occurring in
$[m,m+n)$ up to $[k,k+n)$. More precisely, $r^{*}\in G$ is such
that for all $i\in\omega$ and $\alpha\in\omega_{2}$
\[
r^{*}(i,\alpha)=\begin{cases}
r(i,\alpha) & \text{if }i<m\text{ and }r(i,\alpha)\downarrow,\\
r(i-(k-m),\alpha) & \text{if }i\geq k\text{ and }r(i,\alpha)\downarrow.
\end{cases}
\]
Then we let $s=q\cup r^{*}$ and note that by the filter conditions
$s\in G$. Next let $\sigma:\omega\to\omega$ be the permutation
that swaps the interval $[m,m+n)$ with $[k,k+n)$ and let $\sigma:\mathbb{Q}\cong\mathbb{Q}$
be the $\omega$-based automorphism determined by it. Observe then
that $\sigma(s)\leq r$ and so we see that
\[
\sigma(s)\Vdash u\notin\dot{x}
\]
and then, noting that $\sigma$ is involutive,
\[
s\Vdash u\notin\sigma\dot{x}.
\]
Thus, $u\notin(\sigma\dot{x})_{G}$ and so $(\sigma\dot{x})_{G}\neq x$
as required.
\end{proof}
With this in hand, the result about $SOA$ follows straightforwardly.
\begin{thm}
$SOA$ cannot eliminate imaginaries.\label{thm:SOA=0000ACelim} 
\end{thm}

\begin{proof}
We begin with some setup by letting $G$ be $\mathbb{Q}$-generic
over $L$, $g=\bigcup G:\omega\times\omega_{2}\to2$, and for all
$\alpha<\omega_{2}$, $g_{\alpha}:\omega\to2$ be such that $g_{\alpha}(n)=g(n,\alpha)$.
We then let $\dot{g}_{\alpha}$ be the canonical name for each $g_{\alpha}$.
Next we fix our equivalence relation by letting $E\subseteq\mathbb{R}\times\mathbb{R}$
be mutual constructibility; i.e., we let $xEy$ iff $L[x]=L[y]$.
Note that this is a $\Delta_{2}^{1}$ relation and thus, can be defined
in second order arithmetic. Also note that if $\sigma:\mathbb{Q}\cong\mathbb{Q}$
is a $\omega$-based automorphism from $L$, then for all $\alpha<\omega_{2}$,
$(\dot{g}_{\alpha})_{G}E(\sigma\dot{g}_{\alpha})_{G}$.

It will suffice to show that no function $F:\mathbb{R}\to\mathbb{R}$
can be defined over $L[G]$ such that for all $x,y\in\mathbb{R}$
\[
yEz\leftrightarrow F(y)=F(z).
\]
This will be enough since if such a function cannot be defined in
$L[G]$, then it also cannot be defined in $L[G]$'s version of the
standard model of $SOA$, which is itself is definable within $L[G]$.
We suppose toward a contradiction that there is such a function, $F$.
Since $|L_{\omega_{1}}|^{L[G]}=\omega_{1}$, we may fix $\alpha<\omega_{2}$
such that $F(g_{\alpha})\notin L$. Thus, $F(g_{\alpha})=x$ for some
$x\in\mathbb{R}\cap(N\backslash L)$. Next let $\dot{x}\in L^{\mathbb{P}}$
be such that $\dot{x}_{G}=x$ and fix $p\in G$ such that 
\[
p\Vdash F(\dot{g}_{\alpha})=\dot{x}.
\]
Using Lemma \ref{Karagila}, we may then fix an $\omega$-based automorphism
$\sigma:\mathbb{Q}\cong\mathbb{Q}$ from $L$ such that $\sigma(p)=p$
and $(\sigma\dot{x})_{G}\neq x$. Thus, we obtain 
\[
p\Vdash F(\sigma\dot{g}_{\alpha})=\sigma\dot{x}
\]
and so if we let $g_{\alpha}^{*}=(\sigma\dot{g}_{\alpha})_{G}$ and
$x^{*}=(\sigma\dot{x})_{G}$, we see that $F(g_{\alpha}^{*})=x^{*}$,
$g_{\alpha}Eg_{\alpha}^{*}$ but $x\neq x^{*}$, which is a contradiction.
\end{proof}
It is also worth noting that a softer, albeit less direct proof can
be given for Theorem \ref{thm:SOA=0000ACelim}.\footnote{This observation is due to Ali Enayat and the author independently.}
We start by recalling a couple of facts from the literature. First
we note that if we have a pair of sequential theories that are bi-interpretable
via identity-preserving interpretations, then those theories are definitionally
equivalent \citep{VisserFriedBitoSyn}.\footnote{See Theorem 5.4 in \citep{VisserFriedBitoSyn}.}
We then note that $SOA$ and $ZFC_{count}^{-}$ are: both sequential,
bi-interpretable, but not definitionally equivalent \citep{EnayatLelykFOCat,chenMeadowsTeasing}.
Given Visser and Friedman's theorem, we see that the interpretations
witnessing their bi-interpretability cannot be identity-preserving.
In fact, as we'll discuss in more detail below, $SOA$ interprets
$ZFC_{count}^{-}$ via a quotient interpretation, while $ZFC_{count}^{-}$
does interpret $SOA$ with an identity-preserving interpretation by
simply restricting the domain to the natural numbers and their powerset.
Now suppose toward a contradiction that $SOA$ could eliminate imaginaries.
Then $SOA$ could replace its quotient interpretation of $ZFC_{count}^{-}$
with an identity-preserving interpretation by delivering representatives
for each of the equivalence classes. But then $SOA$ and $ZFC_{count}^{-}$
would be bi-interpretable via identity-preserving interpretations
and Visser and Friedman's theorem would tell us that they were definitionally
equivalent, which we already know is impossible.\footnote{Note that we cannot, so to speak, turn this around and also show that
$ZFC_{count}^{-}$ fails to eliminate imaginaries via a similar indirect
strategy. The reason for this is that $ZFC_{count}^{-}$ already interprets
$SOA$ with an identity-preserving interpretation. Thus, there are
no imaginaries to eliminate.} Nonetheless, it is perhaps valuable to have a more direct proof and
it provides a simple strategic perspective on the somewhat longer
proof below. 

\section{The main theorem\label{sec:The-main-theorem}}

We are now ready to consider $ZFC_{count}^{-}$. To make things precise,
we start with the standard axioms of $ZFC$ based on the Axiom of
Collection rather than Replacement.\footnote{See \citep{GitmanZFC-} for an enlivening discussion of the pathology
that ensues when only the Replacement schema is available.} We then discard the Powerset Axiom and replace it with a new axiom
stating that everything is countable; or more formally, $\forall x\ |x|\leq\omega$.
Before we get down to business, we first make some remarks about why
the theorem appears to be a little more difficult than it may have
at first glance. $ZFC_{count}^{-}$ theory is well-known to be bi-interpretable
with $SOA$. Thus, we might hope that some cosmetic modifications
to the proof above would address our target. This does not appear
to be the case. To see why, let's consider what happens if we try
the obvious thing. The bi-interpretation of $ZFC_{count}^{-}$ and
$SOA$ relies on the fact that every set in a model of $ZFC_{count}^{-}$
can be coded using a well-founded, extensional relation on its version
of $\omega$. Given this, we might hope to modify the proof above
as follows. Using the same equivalence relation $E$ and noting that
$\mathbb{HC}$ is a natural model of $ZFC_{count}^{-}$, we might
suppose toward a contradiction that there is a definable function
$G:\mathbb{R}\to\mathbb{HC}$ that respects $E$. Then we might hope
to find a definable function that $H:\mathbb{HC}\to\mathbb{R}$ that
takes $x\in\mathbb{HC}$ and returns some well-founded, extensional
$R_{x}\subseteq\omega^{2}$ that codes $x$. Then the rest of the
argument would follow essentially the same lines. However, the problem
with this strategy is that there will many different sets $S\subseteq\omega^{2}$
that code $x$. Moreover, it can be seen that there is no uniform
way of defining a function like $H$ in $ZFC_{count}^{-}$.\footnote{This follows from one strategy for showing that $ZFC_{count}^{-}$
and $SOA$ are not definitionally equivalent \citep{EnayatLelykFOCat}.
More particularly, if we start with a model of $V=L$ and then force
to add a Cohen real, $c$, it can be seen that $L[c]$ cannot define
a linear ordering of $\mathbb{HC}^{L[c]}$. But we can clearly define
a linearly ordering the reals of $L[c]$, so $L[c]$ cannot define
an injection from $\mathbb{HC}^{L[c]}$ into $\mathbb{R}^{L[c]}$.
See \citep{chenMeadowsTeasing} for more details.} For these reasons, we will take a more direct route that generalizes
the strategy of Theorem \ref{thm:SOA=0000ACelim}'s proof by generalizing
Karagila's Lemma \ref{Karagila}. 

One might also wonder if we could make do with a less exotic equivalence
relation than mutual constructibility. While there will obviously
be other relations that also work, something in this vicinity seems
to be required. It will be instructive consider why this is the case.
First, let us note that any equivalence relation $E$ on $\mathbb{R}$
whose cells are always sets and not proper classes is doomed to failure.
To see this, note that we can simply let $F$ be such that for all
$x\in\mathbb{R}$, $F(x)=[x]_{E}$ where the latter is just the set
of those $y$ such that $yEx$. For the perhaps canonical example
of this kind,\footnote{Indeed, this was the first thing I tried and the first thing that
was suggested to me when I spoke to others about this problem.} let $\mathbb{R}=2^{\omega}$ and consider the relation $E_{0}$ which
deems reals to be equivalent if they eventually, always agree; i.e.,
$xE_{0}y$ iff there is some $n\in\omega$ such that for all $m>n$,
$x(m)=y(m)$.\footnote{To see that $[x]_{E_{0}}$ is countable, note that it's identical
to $\bigcup_{n\in\omega}D_{n}$ where each $D_{n}$ is the set of
those $y\in\mathbb{R}$ that only differ from $x$ below $n$. } Thus, it is clear that our equivalence relation will need to be a
proper class, which in the context of countable set theory means that
its cells must be uncountable. So that's a natural choice for such
an equivalence relation? Letting $\mathbb{R}=\mathcal{P}(\omega^{2})$,
we might think of reals as coding binary relations on the naturals.
An obvious choice of equivalence relation is then isomorphism between
those relations. Call this $E_{iso}$. Clearly, $[x]_{E_{iso}}$is
uncountable for all $x\in\mathbb{R}$ so we get around the first hurdle,
but there is another.\footnote{I'm grateful to Jason Chen for pointing this out to me.} 
\begin{thm}
($ZFC^{-}$, \citealp{Friedman2000}) There is a definable function
$F_{iso}:\mathbb{R}\to\mathbb{HC}$ such that for all $x,y\in\mathbb{R}$
\[
xE_{iso}y\ \Leftrightarrow\ F_{iso}(x)=F_{iso}(y).
\]
\end{thm}

The proof is so delightfully simple it would be a crime not to state
it. We let $F_{iso}(x)$ be the infinitary sentence $\varphi_{x}$
of $\mathcal{L}_{\omega_{1}\omega}$ known as the Scott sentence for
$x$ construed as a binary relation on $\omega$. The Scott sentence
$\varphi_{x}$ for some $x\in\mathbb{R}$ is naturally coded as an
element of $\mathbb{HC}$ and is such that $y\models\varphi_{x}$
iff $y\cong x$. Moreover, a uniform procedure for defining a Scott
sentence can be devised that is only sensitive to the structural,
or model-theoretic, aspects of some $x\subseteq\omega\times\omega$.
Thus, if $x\cong y$, then we also see that $\varphi_{x}=\varphi_{y}$.\footnote{See \citep{KeislerInf} or \citep{Barwise} for more information about
Scott sentences and the infinitary logic $\mathcal{L}_{\omega_{1}\omega}$.}

So there goes another natural equivalence relation. It's not so clear
what to blame, but it is perhaps worth noting that $E_{iso}$ is relatively
complex as it is an analytic, or $\Sigma_{1}^{1}$, relation that
is not Borel. In contrast, mutual constructibility is significantly
more complex again as a $\Delta_{2}^{1}$ relation. Among other things,
this extra complexity seems to do the trick. 

We are now ready to state and prove the main theorem of this little
paper. 
\begin{thm}
$ZFC_{count}^{-}$ cannot eliminate imaginaries.\label{thm:main}
\end{thm}

Before, we get to the business of proving this, I think it's worth
remarking that -- despite the comments above -- this claim feels
obviously true. Without powerset, Scott's clever trick is gone and
nothing looks likely to replace it. Moreover, I think it's obvious
that some generalization of the proof Theorem \ref{thm:SOA=0000ACelim}
should work. The devil, however, is in the details and, as the reader
will see, there are quite a few of them. With this in mind, I've aimed
to present the proof, so to speak, backwards. I'll start by isolating
a couple of crucial lemmas without proving them. Then I'll show how
Theorem \ref{thm:main} follows relatively quickly from them. The
reader who finds either of those lemma obvious is then welcome to
stop. For the less credulous reader, we'll then prove those lemmas
in the sections that follow. 

For the remainder of this paper we let $\mathbb{P}=Add(\omega,1)$
and $\mathbb{B}=ro(\mathbb{P})$, the complete Boolean algebra known
as the completion of $\mathbb{P}$, which is formed from the regular
open subsets of $\mathbb{P}$. When convenient we shall abuse this
notation and suppose that $\mathbb{P}$ is a dense subset of $\mathbb{B}$,
however, when we do this it will be clear from context or surrounding
remarks. For reals, $x,y\in\mathbb{R}$, we let $x\sim_{c}y$ mean
that $x$ and $y$ are \emph{mutually} \emph{constructible}; i.e.,
$L[x]=L[y]$.\footnote{In the context of the proof of Theorem \ref{thm:SOA=0000ACelim},
we denoted this relation as $E$.} For $x\in\mathbb{R}$, we let
\[
[x]_{c}=\{y\in\mathbb{R}\ |\ y\sim_{c}x\}
\]
and call this a \emph{constructibility degree}. We can now state our
crucial lemmas. 
\begin{lem}
Let $G$ be $\mathbb{P}$-generic over $L$. In $L[G]$, there is
no definable function $F:\mathbb{R}\to L$ such that $F(x)=F(y)$
iff $x\sim_{c}y$.\label{lem:rangeOutOfL}
\end{lem}

In the main proof, we'll work in some $L[G]$ and use this to show
that a putative function $F$ that respects $\sim_{c}$ will need
to send some real outside $L$. In the proof of Theorem \ref{thm:SOA=0000ACelim},
we were able to achieve this a little more easily since we forced
over $L$ using $Add(\omega,\omega_{2})$ and so added more reals
than there were sets in $L_{\omega_{1}}$. 
\begin{lem}
Let $G$ be $\mathbb{B}$-generic over $V$. If $x\in H(\omega_{1})^{V[G]}\backslash V$,
then there is some $\dot{x}\in V^{\mathbb{B}}$ with $\dot{x}_{G}=x$
such that for all $b\in\mathbb{B}$ there is some $\sigma\in Aut(\mathbb{B})^{V}$
such that $\sigma(b)=b$ and $b\Vdash\sigma\dot{x}\neq\dot{x}$.\label{lem:MovingNames}
\end{lem}

This lemma will play a similar role to Karagila's Lemma \ref{Karagila}
in the proof below. We'll use it to establish that a putative function
$F$ that respects $\sim_{c}$ will have an output that is moved by
an automorphism for mutually constructible inputs. This will give
us the contradiction we seek. Here then is the main proof. 
\begin{proof}
(of Theorem \ref{thm:main}) It will suffice to show that there is
a model $N$ of $ZFC_{count}^{-}$ with a definable equivalence relation
for which there is no definable function that respects it. The model
will be obtained by taking the hereditarily countable sets of a generic
extension of $L$ by a Cohen real. More specifically, we let $G$
be $\mathbb{B}$-generic over $L$ and let $N=H(\omega_{1})^{L[G]}$.
We then let our equivalence relation be mutual constructibility, $\sim_{c}$. 

Now working in $N$, we suppose toward a contradiction that $F:\mathbb{R}\to N$
is a function definable over $N$ which is such that for all $x,y\in\mathbb{R}$
\[
x\sim_{c}y\ \Leftrightarrow\ F(x)=F(y).
\]
Lemma \ref{lem:rangeOutOfL}, tells us that we may fix some $d\in\mathbb{R}$
such that $F(d)=x$ for some $x\in N\backslash L$. Then either $d\in L$
or $d\notin L$. If $d\notin L$, then either $L[d]=L[G]$ or $L[d]\subsetneq L[G]$.
To get things moving, let's start by addressing the first of these
as it is most similar to the proof of Theorem \ref{thm:SOA=0000ACelim}
above. After that, we'll complete our case analysis and work through
the rest of the cases. So suppose $L[d]=L[G]$ and let $\dot{d}$
be a $\mathbb{B}$-name such that $\dot{d}_{G}=g$. Since $L[d]=L[G]$,
we may then fix some $b_{0}\in G$ such that $b_{0}\Vdash\dot{d}\sim_{c}\dot{G}$,
where $\dot{G}$ is the canonical name for $G$. Then using Lemma
\ref{lem:MovingNames}, we fix a $\mathbb{B}$-name $\dot{x}$ such
that $\dot{x}_{G}=x$ and for all $b\in\mathbb{B}$ there is some
$\sigma:\mathbb{B}\cong\mathbb{B}$ with $\sigma\in L$ such that
$\sigma(b)=b$ and $b\Vdash\sigma\dot{x}\neq\dot{x}$. We then see
that
\[
L[G]\models F(d)=x\Leftrightarrow\exists b_{1}\in G\ \Vdash F(\dot{d})=\dot{x}.
\]
After fixing such a $b_{1}$, we may then obtain $b\in G$ with $b\leq b_{0},b_{1}$.
Then we may fix an automorphism $\sigma$ of $\mathbb{B}$ that witnesses
the proper described above. We then see that 
\[
b=\sigma(b)\Vdash F(\sigma\dot{d},\sigma\dot{x})\wedge\sigma\dot{d}\sim_{c}\sigma\dot{G}.
\]
Now let $d^{*}=(\sigma\dot{d})_{G_{d}}$ and $x^{*}=(\sigma\dot{x})_{G}$.
Then we $d^{*}\sim_{C}(\sigma\dot{G})_{G}\sim_{c}G\sim_{c}d$. But
since $b\Vdash\sigma\dot{x}\neq\dot{x}$, we also see that $x^{*}\neq x$
and so $F$ does not respect $\sim_{c}$. This completes our warmup
case.

For the rest of our case analysis, note that if $L[d]\subsetneq L[G]$,
it can be seen that there exist $G_{d}$ and $G_{e}$ that are mutually
$\mathbb{B}$-generic over $L$ such that $L[G_{d}]=L[d]$ and $L[G]=L[G_{d}][G_{e}]$.
This leaves us with three more cases to consider. Either: $d\notin L$
and $x\in L[G_{d}]$; $d\notin L$ and $x\in L[G]\backslash L[G_{d}]$;
or $d\in L$. 

Suppose first that $d\notin L$ and $x\in L[G_{d}]$. The proof here
is similar to our warmup case. First we fix $b_{0}\in G_{d}$ such
that $b_{0}\Vdash\dot{d}\sim_{c}\dot{G}_{d}$. Then using Lemma \ref{lem:MovingNames},
fix a $\mathbb{B}$-name $\dot{x}$ such that $\dot{x}_{G_{d}}=x$
and for all $b\in\mathbb{B}$ there is an automorphism $\sigma$ in
$L$ that fixes $b$ while $b$ forces $\sigma\dot{x}\neq\dot{x}$.
Then it can be seen that using the homogeneity properties of $\mathbb{B}$
that 
\begin{align*}
L[G]\models F(d)=x & \Leftrightarrow L[G_{d}]\models``\Vdash F(d)=x"\\
 & \Leftrightarrow\exists b_{1}\in G_{d}\ b\Vdash\Phi(\dot{d},\dot{x})
\end{align*}
where $\Phi(d,x)$ says that $\Vdash F(d)=x$ and $\dot{d}$ is the
canonical name for the Cohen real delivered by $G_{d}$. Having fixed
such a $b_{1}\in G$, we may then obtain $b\in G$ with $b\leq b_{0},b_{1}$.
Then using our special property of $\dot{x}$, we fix an automorphism
$\sigma$ of $\mathbb{B}$ from $L$ such that $\sigma(b)=b$ and
$b\Vdash\sigma\dot{x}\neq\dot{x}$. We then see that 
\[
b=\sigma(b)\Vdash\Phi(\sigma\dot{d},\sigma\dot{x})\wedge\sigma\dot{d}\sim_{c}\sigma\dot{G}_{d}.
\]
Let $d^{*}=(\sigma\dot{d})_{G_{d}}$ and $x^{*}=(\sigma\dot{x})_{G_{d}}$.
Then we see that $L[G]\models F(d^{*})=x^{*}$. Moreover, we see that
$L[d]=L[d^{*}]$ and so $d\sim_{c}d^{*}$. However, since $b\Vdash\sigma\dot{x}\neq\dot{x}$,
we also see that $x^{*}\neq x$ which means that $F$ doesn't respect
$\sim_{c}$. 

Next let us suppose that $d\notin L$ and $x\in L[G]\backslash L[G_{d}]$.
Note that $L[G]$ is a $\mathbb{B}^{*}$-generic extension of $L[G_{d}]$
by $G_{e}$ where $\mathbb{B}^{*}$ is $L[G_{d}]$'s version of $\mathbb{B}=ro(\mathbb{P})$.
Then, using Lemma \ref{lem:MovingNames} -- this time, relative to
$L[G_{d}]$ and $\mathbb{B}^{*}$ rather than $L$ and $\mathbb{B}$
-- we may fix a $\mathbb{B}^{*}$-name $\dot{x}\in L[G_{d}]^{\mathbb{B}^{*}}$
such that $\dot{x}_{G_{e}}=x$ and for all $b\in\mathbb{B}^{*}$ there
is an automorphism $\sigma\in L[G_{d}]$ of $\mathbb{B}^{*}$ that
fixes $b$ while $b\Vdash\sigma\dot{x}\neq\dot{x}$. Then we see that:
\begin{align*}
L[G]\models F(d)=x & \Leftrightarrow L[G_{d}][G_{e}]\models F(d)=x\\
 & \Leftrightarrow\exists b\in G_{e}\ (b\Vdash F(d)=\dot{x})^{L[G_{d}]}.
\end{align*}
Fixing such a $b$ we may then obtain $\sigma\in L[G_{d}]$ such that
$\sigma(b)=b$ and $b\Vdash\sigma\dot{x}\neq\dot{x}$. Then we see
in that in $L[G_{d}]$ we have
\[
b=\sigma(b)\Vdash F(d)=\sigma\dot{x}
\]
and so if we let $x^{*}=(\sigma\dot{x})_{G_{e}}$ we see that $L[G]\models F(d)=x^{*}$,
which is a contradiction since $x^{*}\neq x$. Finally, if we suppose
that $d\in L$, we see that $x\in L[G]\backslash L$ and essentially
the same argument as used in the previous case suffices.
\end{proof}
Thus, we have a relatively natural generalization of the proof of
Theorem \ref{thm:SOA=0000ACelim}. Note that it also follows that
$ZFC^{-}$ cannot eliminate imaginaries either. We are now ready to
prove the lemmas. \ref{lem:Base-1}Moreover, the reader will observe
that the proofs below can also be adapted to show that $ZF^{-}$ plus
the statement that there is an inaccessible cardinal into which every
set can be injected also fails to eliminate imaginaries.\footnote{We discuss the required modifications at the end of the paper.}

\subsection{Lemma \ref{lem:rangeOutOfL}\label{sec:OutofL}}

Here our goal is to show that in a generic extension of $L$ by a
Cohen real any function respecting mutual constructibility will need
to send a real to a set outside $L$. To achieve this, we note that
$L$ is definably well-ordered and then make an automorphism argument
showing that the constructibility degrees in the Cohen extension lack
a definable well-ordering. First, suppose that $G$ is $\mathbb{P}$-generic
over $L$. Then let $0_{c}=\mathbb{R}^{L}=[x]_{c}$ for any $x\in\mathbb{R}\cap L$,
and let $g_{c}=[\bigcup G]_{c}$. The following claim is well-known
but we give a quick proof for self-containment. 
\begin{prop}
(Essentially, \citealp{AbrahamShoreCohenReals}) In $L[G]$, no construcibility
degree other than $0_{c}$ and $g_{c}$ are definable.\label{prop:ConDegDef}
\end{prop}

\begin{proof}
Let us construe $G$ as $G_{0}\times G_{1}$. Then $G_{0}$ and $G_{1}$are
both $\mathbb{P}$-generic over $L$. Let $g_{0},g_{1}$ be $\bigcup G_{0}$
and $\bigcup G_{1}$ respectively. Note that $g_{0}\nsim_{c}g_{1}$
since they are mutually generic. Let $\dot{g}_{0}$ be the canonical
name for $g_{0}$. Now suppose toward a contradiction that $[g_{0}]_{c}$
is definable in $L[G_{0}\times G_{1}]$. Then we may fix a formula
$\varphi(x)$ and $\langle p_{0},p_{1}\rangle\in G_{0}\times G_{1}$
such that 
\[
\langle p_{0},p_{1}\rangle\Vdash\varphi([\dot{g}_{0}]_{c}).
\]
Let $H_{0}\times H_{1}$ formed from $G_{0}\times G_{1}$ letting:
$H_{0}$ be $G_{1}$ except that it agrees with $p_{0}$ on $dom(p_{0})$;
and similarly for $H_{1}$. Essentially, $H_{0}\times H_{1}$ is obtained
by an automorphism of $\mathbb{P}\times\mathbb{P}$ that swaps $G_{0}$
with $G_{1}$ and does some bit-flipping to ensure that $p_{0}\in H_{0}$
and $p_{1}\in H_{1}$. From this we then see that $\langle p_{0},p_{1}\rangle\in H_{0}\times H_{1}$
and so 
\[
L[H_{0}\times H_{1}]\models\varphi([(\dot{g}_{0})_{H_{0}\times H_{1}}]_{c}).
\]
But since $(\dot{g}_{0})_{H_{0}\times H_{1}}$ is a finite variation
of $g_{1}$, we see that $(\dot{g}_{0})_{H_{0}\times H_{1}}\sim_{c}g_{1}$
and since $L[H_{0}\times H_{1}]=L[G_{0}\times G_{1}]$ we see that
\[
L[G_{0}\times G_{1}]\models\varphi([g_{1}]_{c})
\]
and so $\varphi$ does not define $[g_{0}]_{c}$. 
\end{proof}
From this, we easily see that: 
\begin{cor}
In $L[G]$, there is no definable well ordering of its constructibility
degrees.\label{cor:NoDefWOofConsDegs}
\end{cor}

And from here Lemma \ref{lem:rangeOutOfL} then follows quickly. To
see this work in $L[G]$ and note that if there were a definable function
$F:\mathbb{R}\to L$ such that $F(x)=F(y)$ iff $x\sim_{c}y$, then
we could use it to define a well-ordering of the constructibility
degrees by letting $[x]_{c}\triangleleft[y]_{c}$ iff $F(x)<_{L}F(y)$,
contradicting the corollary above. Thus, we see that for any definable
function $F:\mathbb{R}\to L[G]$ with $F(x)=F(y)$ iff $x\sim_{c}y$
there must be some $x\in\mathbb{R}$ such that $F(x)\in L[G]\backslash L$. 

\subsection{Lemma \ref{lem:MovingNames}\label{sec:MovingNames}}

Our goal now is to prove Lemma \ref{lem:MovingNames}. We want to
find names for elements of $\mathbb{HC}$ in a generic extension that
are sufficiently \emph{flexible} that they are easily moved by automorphisms.
These will then be the outputs of the putative function in the proof
by contradiction above. In the basic case of Lemma \ref{Karagila},
we were just dealing with names of subsets of the ground model and
we were able to make do with just one automorphism that moved the
denotation of our name. However, we now have more iterative $\in$-structure
to deal with and this will no longer do. Moreover, our argument below
seems to require that the names be forced to be moved rather than
merely being shifted in a particular generic extension. Addressing
these issues has led to the following definition. In this section,
we let $G$ be $\mathbb{B}$-generic over $V$. 
\begin{defn}
Say that $y\in V[G]$ has a \emph{flexible} \emph{name} $\dot{y}\in V^{\mathbb{B}}$
if $\dot{y}_{G}=y$ and for all $p\in\mathbb{B}$ there is some $\mathcal{A}_{p}\subseteq Aut(\mathbb{B})^{V}$
with $|\mathcal{A}|>\omega$ such that for all $\sigma\neq\tau\in\mathcal{A}_{p}$: 
\begin{itemize}
\item $\sigma(p)=p$; 
\item $p\Vdash\sigma\dot{y}\neq\tau\dot{y}$.
\end{itemize}
All the automorphisms considered below will be of order 2, so we can
largely ignore annoying problems around inverses and thus, keep our
notation relatively clean. Our goal then is to prove the following: 
\end{defn}

\begin{thm}
Every $x\in H(\omega_{1})^{V[G]}\backslash V$ has a flexible name.\label{thm:FlexibleNames}
\end{thm}

It should be clear that Lemma \ref{lem:MovingNames} follows directly
from this claim. It is the heart of our problem. Our plan is to proceed
inductively, so we start by taking a moment to set up a nice ranking
for the sets we are interested in. Working in $V[G]$, we start by
defining a set $A$ inductively by letting:
\begin{itemize}
\item $A_{0}$ be the set of those $x\in H(\omega_{1})$ such that $X\subseteq V$;
\item $A_{\alpha+1}$ be the set of those countable $x$ such that $x\subseteq A_{\alpha}$;
and
\item $A_{\lambda}=\bigcup_{\alpha\in\lambda}A_{\alpha}$ when $\lambda$
is a limit ordinal. 
\end{itemize}
We then note that $A_{\alpha}=A_{\omega_{1}}$ whenever $\alpha\geq\omega_{1}$.
As such, we let $A=A_{\omega_{1}}$ and we let the $A$-rank of a
set $x\in A$ be the least $\alpha$ such that $x\in A_{\alpha}$.
Note that $A=H(\omega_{1})^{V[G]}$. Next we define a set $B$ by
letting:
\begin{itemize}
\item $B_{0}$ be the set of those $x\in H(\omega_{1})$ such that $x\subseteq V$
but $x\notin V$; 
\item $B_{\alpha+1}$ be the set of countable $x$ such that $x\subseteq A_{\alpha}$
and $x\cap B_{\alpha}\neq\emptyset$; and
\item $B_{\lambda}$ be $\bigcup_{\alpha<\lambda}B_{\alpha}$ for limit
ordinals $\lambda$. 
\end{itemize}
As with the $A$-sequence, we see that $B_{\alpha}=B_{\omega_{1}}$
for all $\alpha\geq\omega_{1}$ and so we let $B=B_{\omega_{1}}$.
We note then that $B$ is the set of hereditarily countable sets added
in the generic extension and so $B=H(\omega_{1})^{V[G]}\backslash V$.
The proof of Theorem \ref{thm:FlexibleNames} relies on two technical
lemmas for the base and successor cases. As with our main theorem,
we shall proceed backwards by stating the lemmas (without proof) and
proving Theorem \ref{thm:FlexibleNames} on their basis. The proof
of the lemmas will then occupy us for the remainder of the paper. 
\begin{lem}
If $x\in V[G]\backslash V$ is such that $x\subseteq V$, then $x$
has a flexible name.\label{lem:Base}
\end{lem}

\begin{lem}
If $x\in V[G]\backslash V$ is countable and has a member $y\in x$
with a flexible name, then $x$ has a flexible name.\label{lem:Step}
\end{lem}

\begin{proof}
(of Theorem \ref{thm:FlexibleNames}) We proceed by induction on the
$A$-rank of elements $x$ of $B$. Suppose first that $A$-rank of
$x$ is $0$. Then $x\in V[G]\backslash V$ and $x\subseteq V$, so
Lemma \ref{lem:Base} tells us that $x$ has a flexible name. Suppose
next that the $A$-rank of $x$ is $\alpha+1$. Then we may fix some
$y\in x\cap B_{\alpha}$ and the inductive hypothesis tells us that
$y$ has a flexible name. Then since $x\in V[G]\backslash V$ is countable,
we may use Lemma \ref{lem:Step} to see that $x$ has a flexible name.
Finally, suppose that $A$-rank of $x$ is $\lambda$ for some countable
limit ordinal. Then we see that $x\in B_{\alpha}$ for some $\alpha<\lambda$
and so our inductive hypothesis tells us that $x$ has a flexible
name. 
\end{proof}

\section{The Induction Step Lemma}

We'll start with Lemma \ref{lem:Step} as its proof is somewhat shorter. 
\begin{proof}
Let $\dot{x}_{0}\in V^{\mathbb{B}}$ be such that $(\dot{x}_{0})_{G}=x$
and $\Vdash|\dot{x}_{0}|\leq\omega$. Fix $y\in x$ with a flexible
name $\dot{y}\in V^{\mathbb{B}}$ and let 
\[
\dot{x}=\dot{x}_{0}\cup\{\langle\dot{y},1\rangle\}.
\]
Then note that $\Vdash\dot{y}\in\dot{x}$ and $\Vdash|\dot{x}|\leq\omega$.
We suppose toward a contradiction that $\dot{x}$ is not a flexible
name for $x$. Then we may fix some $b\in\mathbb{B}$ such that for
all $\mathcal{A}\subseteq Aut(\mathbb{B})^{V}$, if every $\sigma\neq\tau\in\mathcal{A}$
is such that $\sigma(b)=b=\tau(b)$ and $b\Vdash\sigma\dot{x}\neq\tau\dot{x}$,
then $|\mathcal{A}|\leq\omega$. This means we may fix a sequence
$\langle\pi_{n}\dot{x}\rangle_{n\in\omega}$ where each $\pi_{n}\in Aut(\mathbb{B})^{V}$
and $\pi_{n}(b)=b$ that is maximal in the sense that if $\sigma\in Aut(\mathbb{B})^{V}$
is such that $\sigma(b)=b$, then there is some $n\in\omega$ such
that $b\nVdash\sigma\dot{x}\neq\pi_{n}\dot{x}$.

Now since $\dot{y}$ is a flexible name, we may fix some $\mathcal{A}_{b}\subseteq Aut(\mathbb{B})^{V}$
with $|\mathcal{A}_{b}|>\omega$ such that for all $\sigma\neq\tau\in\mathcal{A}_{b}$,
$\sigma(b)=b=\tau(b)$ and $b\Vdash\sigma\dot{y}\neq\tau\dot{y}$.
Recalling that we may regard $\mathbb{P}$ is dense in $\mathbb{B}$,
we then observe that for all $\sigma\in\mathcal{A}_{b}$, we may fix
$n_{\sigma}$ and $r_{\sigma}\leq b$ with $r_{\sigma}\in\mathbb{P}$
such that $r_{\sigma}\Vdash\sigma\dot{y}\in\pi_{n_{\sigma}}\dot{x}$.
This follows since we know $\Vdash\dot{y}\in\dot{x}$ and so $\Vdash\sigma\dot{y}\in\sigma\dot{x}$.
Moreover, we may fix some $n\in\omega$ such that $b\nVdash\sigma\dot{x}\neq\pi_{n}\dot{x}$;
and using that, then fix some $r\leq b$ with $r\in\mathbb{P}$ such
that $r\Vdash\sigma\dot{x}=\pi_{n}\dot{x}$. Thus, $r\Vdash\sigma\dot{y}\in\pi_{n}\dot{x}$
as required. 

Then since $|\mathcal{A}_{b}|>\omega=|\mathbb{P}|$ we may fix some
$\mathcal{B}\subseteq\mathcal{A}_{b}$ with $|\mathcal{B}|>\omega$
such that for all $\sigma\in\mathcal{B}$, $r_{\sigma}=r^{*}$ for
some $r\leq b$. But then we have an uncountable $\mathcal{B}$ such
that for all $\sigma\neq\tau\in\mathcal{B}$: 
\begin{enumerate}
\item $r^{*}\Vdash\sigma\dot{y}\neq\tau\dot{y}$; 
\item $r^{*}\Vdash\sigma\dot{y}\in\pi_{n}\dot{x}$ for some $n\in\omega$;
and
\item $r^{*}\Vdash|\pi_{n}\dot{x}|\leq\omega$. 
\end{enumerate}
This is a contradiction.
\end{proof}

\section{The Base Case Lemma}

While the proof of Lemma \ref{Karagila} is relatively simple and
elegant, the generalization that seems to be required for Lemma \ref{lem:Base}
is more torturous, although its underlying strategy remains the same.
This time we want enough automorphims to force uncountably many shifted
denotations of some name. A natural strategy presents itself. We will
build a perfect tree of automorphisms by defining them inductively
along initial segments of $2^{<\omega}$. This has the potential to
give us continuum many automorphism and thus satisfy the requirements
of Lemma \ref{lem:Base}. However, as the proof of Lemma \ref{lem:Step}
illustrates, it is important that these shifted names are forced to
be distinct rather than just being distinct in some generic extension.
This has made it more convenient to shift to the realm of complete
Boolean algebras where automorphisms are much more plentiful. Before
we prove Lemma \ref{lem:Base}, we first recall some facts about Boolean
algebras that give us a lifting property that plays an important role
in the main proof. 

\subsection{A lifting fact about Boolean algebras}

Our goal is to show that any automorphism of a countable atomless
dense subalgebra of $\mathbb{B}$ can be lifted to $\mathbb{B}$.
This will be helpful as our proof below defines such a subalgebra
of $\mathbb{B}$ and we want to make use of automorphisms upon it.
I'll note that I think the little sequence below will be relatively
obvious to the folk, but I haven't found it, or something sufficiently
like it, recorded in the literature. First, we recall Sikorski's extension
theorem. 
\begin{fact}
(Sikorski) Let $\mathbb{A}$ be a subalgebra of $\mathbb{A}^{*}$
and $f:\mathbb{A}\to\mathbb{B}$ be an injective homomorphism where
$\mathbb{B}$ is a complete Boolean algebra. Then $f$ can be extended
to an injective homomorphism $f^{*}:\mathbb{A}^{*}\to\mathbb{B}$
where $f^{*}\restriction\mathbb{A}=id$.\footnote{See Theorem 5.9 and Lemma 5.11 in \citep{KoppBool}.}
\end{fact}

\begin{lem}
Suppose $\mathbb{C}$ is a countable atomless Boolean algebra that
is dense in $\mathbb{B}$. Let $\mathbb{D}=ro(\mathbb{C})$ and $i:\mathbb{C}\to\mathbb{D}$
be the canonical embedding. Then there is some $j:\mathbb{D}\cong\mathbb{B}$
such that $j\circ i=id$.\label{lem:DtoB}
\end{lem}

\begin{proof}
Using Sikorski's extension theorem, we get an embedding $j:\mathbb{D}\to\mathbb{B}$
such that $j\circ i=id$. It will suffice to show that $j$ is a surjection.
We do this in stages. First we claim that $j$ is a dense embedding.
To see this let $b\in\mathbb{B}$. Then fix $c\in\mathbb{C}$ such
that $c\leq b$. Then we see that $j\circ i(c)\leq b$, so we have
$i(c)\in\mathbb{D}$ with $j(i(c))\leq b$ as required. Since $j$
is a dense embedding between complete Boolean algebras, we see that
$j$ is a complete embedding.\footnote{See Lemma VII.8 and Exercise VII.C7 in \citep{KunenST}.}
To see that $j$ is surjection, let $b\in\mathbb{B}$. Then since
$j$ is dense and $\mathbb{D}$ is complete we see that if we let
\[
X=\{d\in\mathbb{D}\ |\ j(d)\leq b\}
\]
is dense below $b$ and so $b=\bigvee j[X]=j(\bigvee X)$ where $\bigvee X\in\mathbb{D}$.
\end{proof}
\begin{lem}
Suppose $\mathbb{C}$ is a countable atomless Boolean algebra. Let
$\mathbb{D}=ro(\mathbb{C})$ and $i:\mathbb{C}\to\mathbb{D}$ be the
canonical embedding. Then whenever $f:\mathbb{C}\cong\mathbb{C}$
there is some $g:\mathbb{D}\cong\mathbb{D}$ such that for all $c\in\mathbb{C}$,
$f(c)=i^{-1}\circ g\circ i(c)$.\label{lem:liftCtoD}
\end{lem}

\begin{proof}
Recall that $\mathbb{D}$ consists of the regular open subsets of
$\mathbb{C}$. More specifically, $A\in\mathbb{D}$ iff $A\subseteq\mathbb{C}$
is such that:
\begin{itemize}
\item $A$ is open; i.e., for all $c\in A$ if $d\leq c$, then $d\in A$;
and
\item $A$ is regular; i.e., if $c\in\mathbb{C}$ is such that $A$ is dense
below $c$, then $c\in A$.\footnote{This is essentially the ``cut'' definition used in Theorem 7.13
of \citep{JechST}.}
\end{itemize}
Moreover, $i:\mathbb{C}\to\mathbb{D}$ is such that $i(c)=c\downarrow$.\footnote{I.e., $c\downarrow=\{c^{*}\in\mathbb{C}\ |\ c^{*}\leq c\}$.}
For $A\in\mathbb{D}$, we let $g(A)=f[A]$. We claim that $g$ is
an an isomorphism. To illustrate this, we just show that $g(A)$ is
a regular open set in $\mathbb{C}$ and leave the rest to the reader.
If $A\in\mathbb{C}$, then $A$ is regular open. We claim that $f[A]$
is regular open. To see that $f[A]$ is open, suppose that $c\in f[A]$
and that $d\leq c$. Then we see that $f^{-1}(d)\leq f^{-1}(c)$ and
so since $f^{-1}(c)\in A$ and $A$ is open, we have $f(d)\in A$.
Thus, $d\in f[A]$ as required. To see that $f[A]$ is regular, suppose
$c\in\mathbb{C}$ is such that of all $d\leq c$ there is some $e\leq d$
with $e\in f[A]$. It will suffice to show that $c\in f[A]$, or in
other words, that $f^{-1}(c)\in A$. Let $d\leq f^{-1}(c)$, so we
have $f(d)\leq c$. Then by our assumption we may fix some $e\leq f(d)$
such that $e\in f[A]$. But then we see that $f^{-1}(e)\leq d$ and
so we've established that $A$ is dense below $f^{-1}(c)$ which means
by the regularity of $A$ that $f^{-1}(c)\in A$ as required. 
\end{proof}
\begin{lem}
Suppose $\mathbb{C}$ is a countable atomless Boolean algebra that
is dense in $\mathbb{B}$. Then every $f:\mathbb{C}\cong\mathbb{C}$
can be lifted to $\mathbb{B}$ to give some $\sigma_{f}:\mathbb{B}\cong\mathbb{B}$
such that $\sigma_{f}\restriction\mathbb{C}=f$.\label{lem:liftCtoB}
\end{lem}

\begin{proof}
Let $\mathbb{D}=ro(\mathbb{C})$ and $i:\mathbb{C}\to\mathbb{D}$
be the canonical embedding. Using Lemma \ref{lem:DtoB} fix $j:\mathbb{D}\cong\mathbb{B}$
such that $j\circ i=id$. Then use Lemma \ref{lem:liftCtoD} to obtain
$g:\mathbb{D}\cong\mathbb{D}$ such that for all $c\in\mathbb{C}$,
$f(c)=i^{-1}\circ g\circ i(c)$. Let $\sigma_{f}:\mathbb{B}\cong\mathbb{B}$
be such that for all $b\in\mathbb{B}$, $\sigma_{f}(b)=j^{-1}\circ g\circ j$.
Then we see that for all $c\in\mathbb{C}$
\begin{align*}
f(c) & =i^{-1}\circ g\circ i(c)\\
 & =i^{-1}\circ j^{-1}\circ\sigma_{f}\circ j\circ i(c)=\sigma_{f}(c).
\end{align*}
\end{proof}

\subsection{Proving Lemma \ref{lem:Base}}

We are now ready for the final piece of the proof of Theorem \ref{thm:FlexibleNames},
which will also complete the proof of our main Theorem \ref{thm:main}.
For a little technical notation, for $b\in\mathbb{B}$ we write $\neg^{0}b$
to denote $b$ and $\neg^{1}b$ to denote $\neg b$. We begin by stating
a lemma from which Lemma \ref{lem:Base} obviously follows. 
\begin{lem}
Let $\dot{x}\in V^{\mathbb{B}}$ be such that $\Vdash\dot{x}\subseteq\check{V}\wedge\dot{x}\notin\check{V}$
and every element of $\dot{x}$ is of the form $\langle\check{n},c\rangle\in\check{V}\times\mathbb{B}$.
Then for all $b\in\mathbb{B}$ there is some $\mathcal{A}\subseteq Aut(\mathbb{B})^{V}$
with $|\mathcal{A}|>\omega$ such that for all $\sigma\neq\tau\in\mathcal{A}$,
$\sigma(b)=b=\tau(b)$ and $b\Vdash\sigma\dot{x}\neq\tau\dot{x}$.\label{lem:Base-1}
\end{lem}

\begin{proof}
Without loss of generality, we may assume that $b=\top$. Using
Lemma \ref{lem:liftCtoB}, it will then suffice to define a countable
atomless Boolean algebra $\mathbb{C}$ that is dense in $\mathbb{B}$
and which is generated by a set of elements $\{b_{n}\}_{n\in\omega}$
such that for all $f,g:\omega\to2$ 
\[
\Vdash\sigma_{f}\dot{x}\neq\sigma_{g}\dot{x}
\]
where $\sigma_{f}:\mathbb{B}\cong\mathbb{B}$ is such that $\sigma_{f}b_{2n}=\neg^{f(n)}b_{2n}$
and $\sigma_{g}$ is defined similarly. Thus, we are concerned with
automorphisms of $\mathbb{C}$ that flip the even elements of the
independent generator set. It will suffice to merely define the set
$\{b_{n}\}_{n\in\omega}$ from which such a $\mathbb{C}$ can be generated.
We shall do this inductively, so that: on even steps we split to ensure
disagreement about membership in $\dot{x}$; and on odd steps, we
work toward ensuring density in $\mathbb{B}$.

Let $\langle p_{n}\rangle_{n\in\omega}$ enumerate $\mathbb{P}$ which
is a dense subset of $\mathbb{B}$. For the case where $n=0$, we
start by fixing $k_{0}\in V$ such that $\nVdash k_{0}\in\dot{x}$
and $\nVdash k_{0}\notin\dot{x}$. Such a $k_{0}$ is guaranteed to
exist since $\Vdash\dot{x}\notin\check{V}$. We then let $b_{0}=\llbracket k_{0}\in\dot{x}\rrbracket$.
For the case where $n=1$, we note that $\{b_{0},\neg b_{0}\}$ is
a maximal anti-chain, so at least one of its members is compatible
with $p_{0}$. For definiteness, suppose it's $b_{0}$. Then we let
$b^{*}\leq p_{0}$ be such that $b^{*}<b_{0}$. Then let $c^{*}<\neg b_{0}$.
We then let $b_{1}=(b_{0}\wedge b^{*})\vee(\neg b_{0}\wedge c^{*})$.
Note that $b_{0}$ and $b_{1}$ are independent of each other. 

For the inductive steps, our goal is to replicate the moves above
as we define what is essentially an infinite binary tree. The definitions
are a little fussy, so it will probably be helpful to bear the base
cases in mind. First, we introduce a little technical notation. For
$b\in\mathbb{B}$, let 
\[
X(b)=\{\langle u,1\rangle\ |\ b\Vdash u\in\dot{x}\}\cup\{\langle u,0\rangle\ |\ b\Vdash u\notin\dot{x}\}.
\]
Thus, $X(b)$ provides a means of recording the sets whose membership
in $\dot{x}$ has been decided one way or another by $b$. Let us
then write $X(b)\curlywedge X(c)$ iff there is some $u$ and $i\in2$
such that $\langle u,i\rangle\in X(b)$ and $\langle u,1-i\rangle\in X(c)$.
Thus, $X(b)\curlywedge X(c)$ holds when $b$ and $c$ decide a membership
fact about $\dot{x}$ in different ways. Let us also say that $s:n\to2$
is an \emph{odd-fixing }sequence, abbreviated $s:n\to_{odd}2$ if
for all odd $m<n$, $s(m)=0$. The idea here is that for all odd $m<n$,
$\neg^{s(m)}b_{m}=b_{m}$, so odd elements of our sequence are fixed.

Let's get to work. Suppose that we've already defined an independent
sequence $\{b_{m}\}_{m<n}$ with the following properties:
\begin{enumerate}
\item If $s\neq t:n\to_{odd}2$, then 
\[
X(\bigwedge_{i<n}\neg^{s(i)}b_{i})\curlywedge X(\bigwedge_{i<n}\neg^{s(i)}b_{i});\text{ and}
\]
\item For all odd $m<n$ where $m=2k+1$, there is some $s:m\to2$ such
that 
\[
\bigwedge_{i<(m+1)}\neg^{s(i)}b_{i}\leq p_{k}.
\]
\end{enumerate}
The first clause is designed to ensure that every way of flipping
even coordinates of $\{b_{m}\}_{m<n}$ gives a different extension
for $\dot{x}$. The second clause is intended to ensure that the completed
sequence generates an algebra that is dense in $\mathbb{B}$. We now
show show that we can find $b_{n}$ such that $\{b_{m}\}_{m<(n+1)}$
has analogous properties. 

First suppose that $n$ is even. Then we note that for all $s:n\to2$,
we may fix $k_{s}\in V$ and $c_{s}\in\mathbb{B}$ such that
\[
\bigwedge_{i<n}\neg^{s(i)}b_{i}\wedge c_{s}\Vdash k_{s}\in\dot{x}
\]
and 
\[
\bigwedge_{i<n}\neg^{s(i)}b_{i}\wedge\neg c_{s}\Vdash k_{s}\notin\dot{x}.
\]
Such, $k_{s}$ and $c_{s}$ exist since $\Vdash\dot{x}\notin\check{V}$.
We then let 
\[
b_{n}=\bigvee_{s:n\to2}(\bigwedge_{i<m}\neg^{s(i)}b_{i}\wedge c_{s})
\]
and note that 
\[
\neg b_{n}=\bigvee_{s:n\to2}(\bigwedge_{i<m}\neg^{s(i)}b_{i}\wedge\neg c_{s}).
\]
It can then be seen that $\{b_{m}\}_{m<(n+1)}$ is an independent
sequence satisfying (1). 

Suppose next that $n$ is odd where $n=2k+1$. Then it is not difficult
to see that 
\[
\{\bigwedge_{i<n}\neg^{s(i)}b_{i}\ |\ s:n\to2\}
\]
is a maximal anti-chain in $\mathbb{B}$ and so for some $s:n\to2$,
$\bigwedge_{i<n}\neg^{s(i)}b_{i}$ is compatible with $p_{k}$. Fix
such an $s$. Then for all $t:n\to2$: if $t\neq s$ fix $c_{t}\in\mathbb{B}$
such that $c_{t}<\bigwedge_{i<n}\neg^{t(i)}b_{i}$; and if $t=s$,
fix $c_{t}<\bigwedge_{i<n}\neg^{t(i)}b_{i}\wedge p_{k}$. Let 
\[
b_{n}=\bigvee_{t:n\to2}\bigwedge_{i<n}(\neg^{t(i)}b_{i}\wedge c_{t}).
\]
Then $\{b_{m}\}_{m<(n+1)}$ is an independent sequence satisfying
(2). We can then see that $\{b_{n}\}_{n\in\omega}$ freely generates
a countable atomless Boolean algebra $\mathbb{C}$ that is dense in
$\mathbb{B}$. Thus, it remains to show that whenever $f\neq g:\omega\to2$,
$\Vdash\sigma_{f}\dot{x}\neq\sigma_{g}\dot{x}$. We do this with a
density argument. Thus, for $f\neq g:\omega\to2$, we show that the
set of $b\in\mathbb{B}$ such that $b\Vdash\sigma_{f}\dot{x}\neq\sigma_{g}\dot{x}$
is dense in $\mathbb{B}$.

Suppose $e\in\mathbb{B}$. Then by our construction, we may fix some
$n\in\omega$ and $t:n\to2$ such that 
\[
\bigwedge_{i<(n+1)}\neg^{t(i)}b_{i}\leq e.
\]
Now we name a bunch of things, so that the argument goes through easily.
First, we fix $m\in\omega$ such that $f$ and $g$ agree up to $m$
but no further; i.e., i.e., $f\restriction m=g\restriction m$ but,
say, $f(m)=0\neq1=g(m$). Then we let:
\begin{itemize}
\item $k=max(2m+1,n)$; 
\item $t^{*}:k\to2$ be such that for all $i<k$: $t^{*}(i)=t(i)$ if $i<n$;
and $t(i)=0$ otherwise; 
\item $f^{*}:k\to2$ be such that for all $i\in\omega$: $f^{*}(i)=f(\frac{i}{2})$
if $i<k$ is even; and $f(i)=0$ otherwise; 
\item $g^{*}:\omega\to2$ be defined analogously; 
\item $s_{*}=f^{*}\restriction(2m+1)=g^{*}\restriction(2m+1)$; and
\item $t_{*}=t^{*}\restriction(2m+1)$. 
\end{itemize}
Note that $f^{*}$, $g^{*}$ and $s_{*}$ are all odd-preserving.
Now observe that the construction above ensures that: 
\[
\bigwedge_{i<2m}\neg^{s_{*}\circ t_{*}(i)}b_{i}\wedge b_{2m}\Vdash k_{s_{*}\circ t_{*}}\in\dot{x}\ \text{and}\ \bigwedge_{i<k}\neg^{s_{*}\circ t_{*}(i)}b_{i}\wedge\neg b_{2m}\Vdash k_{s_{*}\circ t_{*}}\notin\dot{x}
\]
and so 
\[
\bigwedge_{i<k}\neg^{f^{*}\circ t^{*}(i)}b_{i}\Vdash k_{s_{*}\circ t_{*}}\in\dot{x}\ \text{and}\ \bigwedge_{i<k}\neg^{g^{*}\circ t^{*}(i)}b_{i}\Vdash k_{s_{*}\circ t_{*}}\notin\dot{x}.
\]
Moreover, we see that 
\[
\sigma_{f}(\bigwedge_{i<k}\neg^{f^{*}\circ t^{*}(i)}b_{i})=\bigwedge_{i<k}\neg^{t^{*}(i)}b_{i}=\sigma_{g}(\bigwedge_{i<k}\neg^{g^{*}\circ t^{*}(i)}b_{i}).
\]
Thus, 
\[
\bigwedge_{i<k}\neg^{t^{*}(i)}b_{i}\Vdash k_{s_{*}\circ t_{*}}\in\sigma_{f}\dot{x}\ \wedge\ k_{s_{*}\circ t_{*}}\notin\sigma_{g}\dot{x}
\]
which suffices since $\bigwedge_{i<k}\neg^{t^{*}(i)}b_{i}\leq e$.
\end{proof}
With this complete, we see that Lemma \ref{lem:Base} follows, as
does Lemma \ref{thm:FlexibleNames}, and finally, Theorem \ref{thm:main}.
Thus, we see that $ZFC_{count}^{-}$ cannot eliminate imaginaries.
Finally, we tie off a couple of loose ends left above. The first follows
immediately from the proof of Theorem \ref{thm:main}.
\begin{cor}
$ZFC^{-}$ cannot eliminate imaginaries. 
\end{cor}

The second is a mild generalization of the work done above.
\begin{thm}
The theory $ZFC^{-}$ plus there is inaccessible $\kappa$ such that
$\forall x\ |x|\leq\kappa$ cannot eliminate imaginaries. 
\end{thm}

Since the proof strategy is essentially the same, we just sketch the
required modifications. 
\begin{proof}
The main change is that we assume that there is an inaccessible cardinal
$\kappa$ and then force with $Add(\kappa,1)$ instead of $Add(\omega,1)$.
The key point then is that inaccessible cardinals behave sufficiently
like $\omega$ for the arguments below to naturally transfer into
the new setting.\footnote{Indeed, we only make use of regularity.}
We make a few remarks about this. For Proposition \ref{prop:ConDegDef},
the proof adapts straightforwardly using subsets of $\kappa$ in place
of subsets of $\omega$. For the inductive definition below Theorem
\ref{thm:FlexibleNames}, the induction now takes $\kappa$ rather
than $\omega$-steps. We change the statement of Lemma \ref{lem:Step}
by supposing we have a set of cardinality $<\kappa$ rather than $<\omega_{1}$.
The obvious modifications of the proof then work as before. Lemma
\ref{lem:liftCtoD} concerns a fact about lifting embeddings from
countable atomless Boolean algebras to their completions. It is not
difficult to see that this lifting fact holds between arbitrary Boolean
algebras and their completions. However, rather than working with
(1) the countable atomless Boolean algebra, and (2) its isomorphic
variant, the Boolean algebra freely generated from countably many
atoms, we work with ($1^{*}$) the $\kappa$-sized atomless Boolean
algebra closed under $<\kappa$ meets and joins, and (2) its isomorphic
variant, the Boolean algebra based on the Lindenbaum algebra for the
infinitary propositional logic $\mathcal{L}_{\kappa}$ with $\kappa$
many atomic propositions.\footnote{See \citep{Barwise,farahExtAlg} for discussion of this logic.}
The proofs then transfer directly. Finally, we adapt the proof of
Lemma \ref{lem:Base-1} by building a tree based on functions $f:\kappa\to2$
rather than $f:\omega\to2$. The original proof relies on the fact
that we may take a countable dense subset of $\mathbb{B}$. In the
inaccessible context we use a $\kappa$ -sized dense subset of $\kappa$
in its place.
\end{proof}
\bibliographystyle{plainnat}

\end{document}